\documentclass{amsart}
\usepackage{amsfonts}
\usepackage{graphicx}
\usepackage{amscd}
\usepackage{amsmath}
\usepackage{amssymb}
\makeatletter
\@namedef{subjclassname@2010}{%
  \textup{2010} Mathematics Subject Classification}
\makeatother

\theoremstyle{plain}
\newtheorem*{acknowledgement}{Acknowledgement}

\newtheorem{corollary}{\bf Corollary}
\newtheorem{definition}{\bf Definition}
\newtheorem{lemma}{\bf Lemma}

\newtheorem{proposition}{\bf Proposition}
\newtheorem{remark}{Remark}

\newtheorem{theorem}{\bf Theorem}
\newcommand{\ric}{\mathring{Ric}}
\newcommand{\hess}{\mathring{Hessf}}
\newcommand{\rc}{\mathring{R}}

\theoremstyle{definition}
\newtheorem{example}{\bf Example}

\numberwithin{equation}{section}

\begin{document}

\title[Critical metrics of the volume functional]{Critical metrics of the volume functional on compact three-manifolds with smooth boundary}

\author{R. Batista}
\author{R. Di\'ogenes}
\author{M. Ranieri}
\author{E. Ribeiro Jr}

\address[R. Batista]{Universidade Federal do Piau\'{\i} -UFPI, Departamento de Matem\'{a}tica, 64049-550 Te\-re\-si\-na - PI, Brazil.}
\email{rmarcolino@ufpi.edu.br}

\address[R. Di\'ogenes]{Universidade da Integra\c c\~ao Internacional da Lusofonia Afro-Brasileira - UNILAB, Instituto de Ci\^encias Exatas e da Natureza, 62785-000 - Acarape - CE, Brazil}
\email{rafaeldiogenes@unilab.edu.br}

\address[M. Ranieri]{Universidade Federal do Cear\'a - UFC, Departamento  de Matem\'atica, Campus do Pici, Av. Humberto Monte, Bloco 914,
60455-760, Fortaleza - CE, Brazil}
\email{ranieri2011@gmail.com}

\address[E. Ribeiro Jr]{Universidade Federal do Cear\'a - UFC, Departamento  de Matem\'atica, Campus do Pici, Av. Humberto Monte, Bloco 914,
60455-760, Fortaleza - CE, Brazil.}
\email{ernani@mat.ufc.br}
\thanks{E. Ribeiro and R. Batista were partially supported by CNPq / Brazil}
\thanks{M. Ranieri was partially supported by CAPES / Brazil}
\date{March 12, 2016}

\newcommand{\spacing}[1]{\renewcommand{\baselinestretch}{#1}\large\normalsize}
\spacing{1.3}

\begin{abstract}
We study the space of smooth Riemannian structures on compact three-manifolds with boundary that satisfies a critical point equation associated with a boundary value problem, for simplicity, Miao-Tam critical metrics. We provide an estimate to the area of the boundary of Miao-Tam critical metrics on compact three-manifolds. In addition, we obtain a B\"ochner type formula which enables us to show that a Miao-Tam critical metric on a compact three-manifold with positive scalar curvature must be isometric to a geodesic ball in $\Bbb{S}^3.$
\end{abstract}

\maketitle

\section{Introduction}
\label{intro}
 In the last decades very much attention has been given to characterizing critical metrics of the Riemannian functionals, as for instance, the total scalar curvature functional and the volume functional. Einstein and Hilbert proved that the critical points of the total scalar curvature functional restricted to the set of smooth Riemannian structures on $M^n$ of unitary volume are Einstein (cf. Theorem 4.21 in \cite{besse}). It permits to prove the existence, on a given compact manifold, of many metrics with constant scalar curvature. This stimulated many interesting works. In this spirit, Miao and Tam \cite{miaotam,miaotamTAMS}, motivated also by a result obtained in \cite{fan}, studied variational properties of the volume functional constrained to the space of metrics of constant scalar curvature on a given compact manifold with boundary.  Some results obtained in \cite{miaotam} suggest that critical metrics with a prescribed boundary metric
seem to be rather rigid.

Following the terminology used in \cite{BDR} we recall the definition of Miao-tam critical metrics.

\begin{definition}
\label{def1} A Miao-Tam critical metric is a 3-tuple $(M^n,\,g,\,f),$ where $(M^{n},\,g)$ is a compact Riemannian manifold of dimension at least three with a smooth boundary $\partial M$ and $f: M^{n}\to \Bbb{R}$ is a smooth function such that $f^{-1}(0)=\partial M$ satisfying the overdetermined-elliptic system
\begin{equation}
\label{eqMiaoTam} \mathfrak{L}_{g}^{*}(f)=g.
\end{equation} Here, $\mathfrak{L}_{g}^{*}$ is the formal $L^{2}$-adjoint of the linearization of the scalar curvature operator $\mathfrak{L}_{g}$. Such a function $f$ is called a potential function.
\end{definition}

It is worthwhile to remark that $$\mathfrak{L}_{g}^{*}(f)=-(\Delta f)g+Hess f-fRic,$$ where $Ric,$ $\Delta$ and $Hess $ stand, respectively, for the Ricci tensor, the Laplacian operator and the Hessian form on $M^n;$ for more details see, for instance, \cite{besse} (cf. Eq. (1.183), p. 64). Therefore, (\ref{eqMiaoTam}) can be rewritten as
\begin{equation}\label{eqMiaoTam1}
-(\Delta f)g+Hess f-fRic=g.
\end{equation}

Miao and Tam \cite{miaotam} were able to show that these critical metrics arise as critical points of the volume functional on $M^n$ when restricted to the class of metrics $g$ with prescribed constant scalar curvature such that $g_{|_{T \partial M}}=h$ for a prescribed Riemannian metric $h$  on the boundary. Afterward, Corvino, Eichmair and Miao \cite{CEM} studied the modified problem of finding stationary points for the volume functional on the space of metrics whose scalar curvature is equal to a given constant. From Theorem 7 in \cite{miaotam}, connected Riemannian manifolds satisfying (\ref{eqMiaoTam1}) have constant scalar curvature $R.$ Some explicit examples of Miao-Tam critical metrics can be found in \cite{miaotam}. Those examples were obtained on connected domain with compact closure in $\Bbb{R}^n,$ $\Bbb{H}^n$ and $\Bbb{S}^n.$ For the sake of completeness, it is very important to underline ones. 

Let us start with the Euclidean space $(\Bbb{R}^{n},\,g),$ where $g$ is its canonical metric. 

\begin{example}[\cite{miaotam}]
We consider $\Omega$ to be a Euclidean ball in $\Bbb{R}^n$ of radius $r.$ Suppose that $$f=-\frac{1}{2(n-1)}|x|^{2}+\frac{1}{2(n-1)}r^{2},$$ where $x\in \Bbb{R}^n.$ Under these conditions, it is not hard to check that $(\Omega,\,g,\,f)$ is a Miao-Tam critical metric. 
\end{example}

At the same time, we present a similar example as before on the standard sphere $(\Bbb{S}^{n},\,g_{0}),$ where $g_{0}$ is its canonical metric. 

\begin{example}[\cite{miaotam}]
Let $\Omega$ be a geodesic ball in $\Bbb{S}^n\subset\mathbb{R}^{n+1}$ with radius $r_{0}\neq \frac{\pi}{2}.$ Suppose that $$f=\frac{1}{n-1}\Big(\frac{\cos r}{\cos r_{0}}-1\Big),$$ where $r$ is the geodesic distance from the point $(0,...,0,1).$ Therefore, $f=0$ on the boundary of $\Omega$ and $f$ satisfies (\ref{eqMiaoTam1}). Moreover, if $\Omega$ is contained in a hemisphere, then $(\Omega,\,g_{0},\,f)$ is also a Miao-Tam critical metric. 
\end{example}

Reasoning as in the spherical case it is not difficult to build a similar example on the hyperbolic space $\Bbb{H}^n.$ 

\begin{example}[\cite{miaotam}]
Regarding the hyperbolic space $\Bbb{H}^n$ embed in $\Bbb{R}^{n,1},$ the Minkowski space, with standard metric $dx_{1}^{2}+...+dx_{n}^{2}-dt^{2}$ such that $$\Bbb{H}^{n}=\{(x_{1},...,x_{n},t)\,|\, x_{1}^{2}+...+x_{n}^{2}-t^{2}=-1,\,t>0\}.$$ We assume that $\Omega$ is a geodesic ball in $\Bbb{H}^n$ with center at $(0,..., 0,1)$ and geodesic radius $r_{0}$. Suppose that $$f=\frac{1}{n-1}\Big(1-\frac{\cosh r}{\cosh r_{0}}\Big),$$ where $r$ is the geodesic distance form the point $(0,...,0,1).$ Similarly, we have that $f=0$ on the boundary of $\Omega$ and $f$ satisfies (\ref{eqMiaoTam1}). For more details see \cite{miaotam}.
\end{example}

It is natural to ask whether these examples are the only Miao-Tam critical metrics. In this sense, inspired by ideas developed by Kobayashi \cite{kobayashi}, Kobayashi and Obata \cite{obata}, Miao and Tam \cite{miaotamTAMS} gave a partial answer to this question. In fact, they proved that a locally conformally flat simply connected, compact Miao-Tam critical metric $(M^{n},g,f)$ with boundary isometric to a standard sphere $\mathbb{S}^{n-1}$ must be isometric to a geodesic ball in a simply connected space form $\mathbb{R}^{n}$, $\mathbb{H}^{n}$ or $\mathbb{S}^{n}.$ This result was improved in dimensions 3 and 4 in \cite{BDR}. More precisely, Barros, Di\'{o}genes and Ribeiro \cite{BDR}, based on the techniques outlined in \cite{CaoChen}, proved that a Bach-flat simply connected, four-dimensional compact Miao-Tam critical metric with boundary isometric to a standard sphere $\mathbb{S}^{3}$ must be isometric to a geodesic ball in a simply connected space form $\mathbb{R}^{4}$, $\mathbb{H}^{4}$ or $\mathbb{S}^{4}.$ In addition, they showed that in dimension three the result even is true replacing the Bach-flat condition by the weaker assumption that $M^3$ has divergence-free Bach tensor.

At the same time, by assuming the Einstein condition, Miao and Tam (cf. Theorem 1.1 in \cite{miaotamTAMS})  were able to remove the condition of boundary isometric to a standard sphere. More precisely, they proved the following result.

\begin{theorem}[Miao-Tam, \cite{miaotamTAMS}]
\label{thmMT}
Let $(M^{n},g,f)$ be a connected, compact Einstein Miao-Tam critical metric with smooth boundary $\partial M.$ Then $(M^{n},g)$ is isometric to a geodesic ball in a simply connected space form $\mathbb{R}^{n}$, $\mathbb{H}^{n}$ or $\mathbb{S}^{n}.$
\end{theorem} Recently, Baltazar and Ribeiro \cite{BalRi} improves Theorem \ref{thmMT} by replacing the assumption of Einstein in the Miao-Tam result by the parallel Ricci tensor condition, which is weaker that the former. See the work of Corvino, Eichmair and Miao \cite{CEM} for further results related.
 
Before to proceed, let us point out that the formal adjoint of the linearization of the scalar curvature plays an important role in problems
related to prescribing the scalar curvature function. We also recall that a complete Riemannian manifold $M^n$ with boundary $\partial M$ (possibly empty) is said to be {\it static} if it admits a smooth non-trivial solution $\lambda\in C^{\infty}(M)$ to the equation
\begin{equation}
\label{eqFM}
 \mathfrak{L}_{g}^{*}(\lambda)=0.
\end{equation} It is well-known that the existence of a static potential imposes many restrictions on the geometry of the underlying manifold. For more details see, for instance, \cite{Wald}.

We recall that Fischer-Marsden \cite{fischermarsden} conjectured that a standard sphere is the only solution to the equation (\ref{eqFM}) on compact manifold. A counter-example to the Fischer-Marsden conjecture was obtained when $g$ is conformally flat, for more details, we refer the reader to \cite{kobayashi} and \cite{lafontaine}. However, a classical result states that the standard hemisphere has the maximum possible boundary area among static three-manifolds with positive scalar curvature and connected boundary. More precisely, a result due to Shen \cite{Shen} and Boucher-Gibbons-Horowitz \cite{BGH} asserts that the boundary $\partial M$ of a compact three-dimensional oriented static manifold with connected boundary and scalar curvature $6$ must be a $2$-sphere whose area satisfies the inequality $|\partial M|\leq 4\pi.$ Moreover, the equality holds if and only if $M^3$ is equivalent to the standard hemisphere. A similar result was obtained by Hijazi, Montiel
and Raulot \cite{HMR}. 

Based on the above result and taking into account that Miao-Tam critical metrics have constant scalar curvature, we shall estimate the area of the boundary of a compact three-manifold satisfying (\ref{eqMiaoTam}). We now state our first result as follows.

\begin{theorem}\label{thmMain}
Let $(M^{3},\,g,\,f)$ be a compact, oriented, Miao-Tam critical metric with connected boundary $\partial M$ and nonnegative scalar curvature. Then $\partial M$ is a $2$-sphere and
\begin{equation}
\label{eqthmMain} |\partial M|\leq\frac{4\pi}{C(R)},
\end{equation} where $C(R)=\frac{R}{6}+\frac{1}{4|\nabla f|^2}$ is constant. Moreover, the equality in (\ref{eqthmMain}) holds if and only if $(M^{3},\,g)$ is isometric to a geodesic ball in a simply connected space form $\Bbb{R}^3$ or $\Bbb{S}^3.$
\end{theorem}

For what follows, we remember that Ambrozio \cite{Lucas} obtained some classification results for static three-manifolds with positive scalar curvature. To do so, he provided a B\"ochner type formula for 
three-dimensional static manifolds in\-vol\-ving the traceless Ricci tensor and the Cotton tensor. For our purposes we recall that  the Weyl tensor $W$ is defined by the following decomposition formula
\begin{eqnarray}
\label{weyl}
R_{ijkl}&=&W_{ijkl}+\frac{1}{n-2}\big(R_{ik}g_{jl}+R_{jl}g_{ik}-R_{il}g_{jk}-R_{jk}g_{il}\big) \nonumber\\
 &&-\frac{R}{(n-1)(n-2)}\big(g_{jl}g_{ik}-g_{il}g_{jk}\big),
\end{eqnarray} where $R_{ijkl}$ stands for the Riemann curvature tensor, whereas the Cotton tensor $C$ is given by
\begin{equation}
\label{cotton} \displaystyle{C_{ijk}=\nabla_{i}R_{jk}-\nabla_{j}R_{ik}-\frac{1}{2(n-1)}\big(\nabla_{i}R
g_{jk}-\nabla_{j}R g_{ik}).}
\end{equation} Note that $C_{ijk}$ is skew-symmetric in the first two indices and trace-free in any two indices. We also remember that  $W \equiv 0$ in dimension three.

In the sequel, motivated by \cite{Lucas}, we provide a B\"ochner type formula for Riemannian manifolds satisfying (\ref{eqMiaoTam1}), which is similar to Ambrozio's formula obtained for static spaces. Let us point out that our arguments designed for the proof of such a formula differ significantly from  \cite{Lucas}. Here, we shall use a formula involving the commutator of the Laplacian and Hessian acting on functions outlined in \cite{VIA}. More precisely, we have established the following result.

\begin{theorem}\label{thmtypebochener}
Let $(M^3,\,g)$ be a connected, smooth Riemannian manifold and $f$ is a smooth function on $M^3$ satisfying (\ref{eqMiaoTam1}). Then we have:
\begin{equation}\label{typebochener}
\frac{1}{2}div\Big(f\nabla|\mathring{Ric}|^2\Big) = \Big(|\nabla \ric|^2 + \frac{|C|^2}{2}\Big)f + \Big(R|\mathring{Ric}|^2  + 6 tr(\mathring{Ric}^3)\Big)f + \frac{3}{2}|\mathring{Ric}|^2,
\end{equation} where $C$ stands for the Cotton tensor and $\mathring{Ric}$ is the traceless Ricci tensor.
\end{theorem}

Finally, as an application of Theorem \ref{thmtypebochener} and Theorem \ref{thmMT} we obtain the following rigidity result.

\begin{corollary}\label{corthmtypebochener}
Let $(M^3,\,g,\,f)$ be a compact, oriented, connected Miao-Tam critical metric with smooth boundary $\partial M$ and positive scalar curvature, $f$ is also assumed to be nonnegative. Suppose that
\begin{equation}
\label{eqCOR}
|\mathring{Ric}|^2 \leq \frac{R^2}{6}.
\end{equation} Then $M^3$ is isometric to a geodesic ball in $\mathbb{S}^3$.
\end{corollary}

\section{Preliminaries}
\label{Preliminaries}

In this section we shall present some preliminaries which will be useful for the
establishment of the desired results. Firstly, we recall that the fundamental equation of a Miao-Tam critical metric (\ref{eqMiaoTam}) becomes
\begin{equation}\label{eq:miaotam}
-(\Delta f)g+Hessf-fRic=g.
\end{equation}  For simplicity, we can rewrite equation (\ref{eq:miaotam}) in the tensorial language as follows
\begin{equation}
\label{fundeqtens} -(\Delta f)g_{ij}+\nabla_{i}\nabla_{j}f -fR_{ij}=g_{ij}.
\end{equation} In particular, tracing (\ref{eq:miaotam}) we have
\begin{equation}\label{eqtrace}
(n-1)\Delta f+Rf=-n.
\end{equation} From this, it is easy to check that
\begin{equation}
\label{p1a}
f\mathring{Ric}=Hess\,f +\frac{Rf+n}{n(n-1)}g,
\end{equation} where $\mathring{T}$ stands for the traceless of $T.$ We also have

\begin{equation}
\label{IdRicHess} f\mathring{Ric}=\mathring{Hess} f.
\end{equation} The following lemma, obtained previously in \cite{BDR}, will be useful.

\begin{lemma}[\cite{BDR}]
\label{lem1BDR} Let $\big(M^n,\,g,\,f)$ be a Miao-Tam critical metric. Then
\begin{equation*}
f\big(\nabla_{i}R_{jk}-\nabla_{j}R_{ik}\big)=R_{ijks}\nabla^{s}f + \frac{R}{n-1}\big(\nabla_{i}f
g_{jk}-\nabla_{j}f g_{ik}\big)- \big(\nabla_{i}f R_{jk}-\nabla_{j}f R_{ik}\big).
\end{equation*}
\end{lemma}

We now recall the following well-known lemma. 

\begin{lemma}
\label{lemdivformula} Let $T$ be a $(0,2)$-tensor on a Riemannian manifold $(M^n,\,g)$ then
\begin{equation*}
div(T(\phi Z))=\phi(divT)(Z)+\phi\langle\nabla Z,T\rangle+T(\nabla\phi,Z),
\end{equation*} for all $Z\in\mathfrak{X}(M)$ and any smooth function $\phi$ on $M.$
\end{lemma}

We also remember that if $\phi:\Sigma\to M$ is an immersion and $h$ denotes the second fundamental form, we have the Gauss Equation:
\begin{equation}
\label{eqgauss}R^{\Sigma}_{ijkl}=R_{ijkl}-h_{il}h_{jk}+h_{ik}h_{jl}.
\end{equation} For more details see \cite{chow} p. 485.

In the sequel we compute the commutator of the Laplacian and Hessian acting on functions. A detailed proof can be found in \cite{VIA} (cf. Proposition 7.1 in \cite{VIA}).

\begin{lemma}\label{commutator}
Let $f\in C^{4}(M).$ Then we have:
\begin{eqnarray*} (\Delta \nabla^2 f)_{ij} & = & \nabla^2_{ij} \Delta f + (R_{jp}g_{ik} + R_{ip}g_{jk} - 2R_{ikjp})\nabla^k\nabla^p f \\
& & + (\nabla_i R_{jp} + \nabla_j R_{pi} - \nabla_p R_{ij})\nabla^p f,
\end{eqnarray*} where $\nabla^2$ also stands for the Hessian.
\end{lemma}

\begin{proof}
Since this lemma is crucial for the establishment of Theorem \ref{thmtypebochener}, we include its proof here for the sake of completeness. To do so, we adopt the notation used in \cite{VIA}. In fact, by using the Ricci identity, it is easy to check that

\begin{eqnarray}
\label{uho}
(\Delta \nabla^2 f)_{ij} & = & g^{kl}\nabla_k\nabla_l\nabla_i\nabla_j f\nonumber \\
& = & g^{kl}\nabla_k[\nabla_i\nabla_l\nabla_j f + R_{lijp}\nabla^pf]\nonumber \\
& = & g^{kl}\nabla_k\nabla_i\nabla_j\nabla_l f + g^{kl}\nabla_kR_{lijp}\nabla^p f + g^{kl}R_{lijp}\nabla_k\nabla^p f\nonumber \\
& = & I + II,
\end{eqnarray} where
\begin{eqnarray*}
I & = & g^{kl}\nabla_k\nabla_i\nabla_j\nabla_l f\\
 & = & g^{kl}[\nabla_i\nabla_k\nabla_j\nabla_l f + R_{kijp}\nabla^p\nabla_l f + R_{kilp}\nabla_j\nabla^p f] \\
& = & g^{kl}[\nabla_i(\nabla_j\nabla_k\nabla_l f + R_{kjlp}\nabla^p f)] + R_{kijp}\nabla^p\nabla^k f + R_{ip}\nabla_j\nabla^p f \\
& = & \nabla_i\nabla_j\Delta f + \nabla_i(R_{jp}\nabla^p f) - R_{ikjp}\nabla^k\nabla^p f + R_{ip}\nabla_j\nabla^p f,
\end{eqnarray*} and the second term is
\begin{eqnarray*}
II & = & g^{kl}\nabla_kR_{lijp}\nabla^p f + g^{kl}R_{lijp}\nabla_k\nabla^p f \\
& = & -g^{kl}\nabla_k R_{jpil}\nabla^p f + R_{lijp}\nabla^l\nabla^p f \\
& = & -(\nabla_pR_{ji} - \nabla_j R_{pi})\nabla^pf - R_{iljp}\nabla^l\nabla^p f\\
& = & -\nabla_pR_{ji}\nabla^p f + \nabla_j R_{pi}\nabla^p f - R_{ikjp}\nabla^k\nabla^p f.
\end{eqnarray*} Here, we used the once contracted Bianchi identity. Rearranging the terms, we obtain from (\ref{uho}) that 

\begin{eqnarray*}
(\Delta \nabla^2 f)_{ij} & = & \nabla_i\nabla_j\Delta f + \nabla_iR_{jp}\nabla^p f + R_{jp}\nabla_i\nabla^p f - R_{ikjp}\nabla^k\nabla^p f + R_{ip}\nabla_j\nabla^p f  \\
& & -\nabla_pR_{ji}\nabla^p f + \nabla_j R_{pi}\nabla^p f - R_{ikjp}\nabla^k\nabla^p f\\
& = & \nabla_i\nabla_j \Delta f + (R_{jp}g_{ik} + R_{ip}g_{jk} - 2R_{ikjp})\nabla^k\nabla^p f \\
& & + (\nabla_i R_{jp} + \nabla_j R_{pi} - \nabla_p R_{ij})\nabla^p f,
\end{eqnarray*} which allows us to complete the proof of the lemma.
\end{proof}

Next, we deduce a formula involving the Cotton tensor (\ref{cotton}) of Miao-Tam critical metrics on three-manifolds. It also plays crucial role in the proof of Theorem \ref{thmtypebochener}.

\begin{lemma}\label{lemacotton}
Let $(M^3,\,g,\,f)$ be a three-dimensional Miao-Tam critical metric. Then we have:
\begin{equation}
f^2|C|^2 = -4fC_{ijk}\nabla^if\mathring{Ric}^{jk}.
\end{equation}
\end{lemma}

\begin{proof} We start invoking Lemma \ref{lem1BDR} (see also Lemma 1 in \cite{BDR}) to get

\begin{equation}
fC_{ijk} = R_{ijkp}\nabla^p f  + \frac{R}{2}(\nabla_i f g_{jk} - \nabla_j f g_{ik}) - (\nabla_i f R_{jk} - \nabla_j f R_{ik})
\end{equation} By using the decomposition of $R_{ijkp}$ we have
\begin{eqnarray}
\label{e56}
fC_{ijk} &=& 2(R_{ik}\nabla_j f - R_{jk}\nabla_i f) - R(g_{ik}\nabla_j f - g_{jk}\nabla_i f)\nonumber\\&& + (g_{ik}R_{jp}\nabla^p f - g_{jk}R_{ip}\nabla^p f).
\end{eqnarray}
Now, substituting the traceless Ricci tensor into (\ref{e56}) we arrive at

\begin{eqnarray}\label{pl1}
fC_{ijk} & = & 2(\rc_{ik} \nabla_j f - \rc_{jk} \nabla_i f) + (\rc_{js}\nabla^s f g_{ik} - \rc_{is}\nabla^s f g_{jk})
\end{eqnarray} Taking into account that $C_{ijk}$ is skew-symmetric in the first two indices we have

\begin{eqnarray*}
-fC_{ijk}\rc^{jk}\nabla^i f & = & -\frac{f}{2}(C_{ijk} - C_{jik})\rc^{jk}\nabla^i f\\
& = & -\frac{f}{2}C_{ijk}\rc^{jk}\nabla^i f +\frac{f}{2}C_{ijk}\rc^{ik}\nabla^j f\\
& = & fC_{ijk}\frac{1}{2}(\rc^{ik}\nabla^j f - \rc^{jk}\nabla^i f).
\end{eqnarray*} Finally, by using this data, together with (\ref{pl1}), and remembering that $C_{ijk}$ is trace-free in any two indices we infer

\begin{eqnarray*}
-fC_{ijk}\rc^{jk}\nabla^i f & = & f C_{ijk}\left(\frac{1}{4}fC^{ijk} -\frac{1}{4}\left(\rc^{js}\nabla_s f g^{ik} - \rc^{is}\nabla_s f g^{jk}\right)\right)\\
 & =& \frac{1}{4}f^2|C|^2.
\end{eqnarray*} and we finish the establishment of the lemma.

\end{proof}

\section{Key Lemmas}

In what follows, we assume that $(M^{n},g,f)$ is a connected, compact  Miao-Tam critical metric with connected boundary $\partial M.$ Under these conditions, since $f^{-1}(0)=\partial M$ we deduce that $f$ does not change of signal. From now on we assume that $f$ nonnegative. In particular, $f>0$  at the interior of $M.$ Moreover, at regular points of $f,$ the vector field $\nu=-\frac{\nabla f}{\mid \nabla f \mid}$ is normal to $\partial M$ and also $|\nabla f|\neq 0$ along $\partial M.$ Therefore, the boundary condition, together with (\ref{eq:miaotam}), implies
\begin{eqnarray*}
X\big(|\nabla f|^{2}\big)&=&2\langle\nabla_{X}\nabla f,\nabla f\rangle\nonumber\\&=&2\nabla^{2}f(X,\nabla f)\nonumber\\&=&-\frac{2}{n-1}\langle X,\nabla f\rangle=0,
\end{eqnarray*} where $X\in\mathfrak{X}(\partial M).$ Whence, $|\nabla f|^{2}$ is constant along $\partial M.$ In particular, we arrive at

\begin{equation}
\nabla_{i}\nabla_{j}f=-\frac{1}{n-1}g_{ij}
\end{equation} along $\partial M.$ 

\begin{remark}
\label{r1}
It is convenient to point out that, choosing appropriate coordinates (e.g. harmonic coordinates) we conclude that $f$ and $g$ are analytic, see for instance Theorem 2.8 in \cite{Corvino} (or Proposition 2.1 in \cite{CEM}). Whence, we conclude that $f$ can not vanish identically in a non-empty open set. Then, the set  of regular points is dense in $M.$
\end{remark}

Proceeding, following the notation used in \cite{miaotam}, the second fundamental form of $\partial M$  is given by
\begin{equation}
\label{secff}
h_{ij}=\langle \nabla_{e_{i}}\nu,e_{j}\rangle,
\end{equation} where $\{e_{1},\ldots,e_{n-1}\}$ is  an orthonormal frame on $\partial M.$ Of which implies that

\begin{equation}
\label{segform}h_{ij}=-\langle \nabla_{e_{i}}\frac{\nabla f}{|\nabla f|},e_{j}\rangle=\frac{1}{(n-1)|\nabla f|}g_{ij}.
\end{equation} Hence, the mean curvature is constant and then $\partial M$ is totally umbilical. For more background see also Theorem 7 in \cite{miaotam}.

In the sequel, as a consequence of Lemma \ref{lemdivformula}, we deduce an integral formula, which is useful in the proof of Theorem \ref{thmMain}. 

\begin{lemma}\label{lemB2}
Let $(M^{n},g,f)$ be a compact, oriented, connected Miao-Tam critical metric with smooth boundary $\partial M$ and nonnegative potential function $f.$ Then we have:
\begin{equation}
\label{eqLb}
 \int_{M}f|\mathring{Ric}|^{2}dM_{g}=-\frac{1}{|\nabla f|}\int_{\partial M}\mathring{Ric}(\nabla f,\nabla f)dS.
\end{equation}

\end{lemma}

\begin{proof}
We start choosing $T=\mathring{Ric}$, $Z=\nabla f$ and $\phi=1$ in Lemma \ref{lemdivformula}. So, since $(M^{n},\,g)$ has constant scalar curvature we infer
\begin{eqnarray*}
div(\mathring{Ric}(\nabla f))&=&(div\mathring{Ric})(\nabla f)+\langle\mathring{Ric},\nabla^2 f\rangle\\
 &=& f|\mathring{Ric}|^2.
\end{eqnarray*} Upon integrating over $M$ we use Stokes formula to obtain
\begin{eqnarray*}
\int_M f|\mathring{Ric}|^2dM&=&\int_M div(\mathring{Ric}(\nabla f))dM\\
 &=&\int_{\partial M}\langle\mathring{Ric}(\nabla f),\nu\rangle dS.
\end{eqnarray*} Next, taking into account that $f\geq 0$ we have $\nu=-\frac{\nabla f}{|\nabla f|}.$ From here it follows that
\begin{eqnarray*}
\int_M f|\mathring{Ric}|^2dM=-\frac{1}{|\nabla f|}\int_{\partial M}\mathring{Ric}(\nabla f,\nabla f)dS,
\end{eqnarray*} where we used that  $|\nabla f|$ is constant on $\partial M.$ This finishes the proof of the lemma.
\end{proof}

\begin{remark}
\label{r2}
Notice that if we replace the nonnegative potential function condition by the nonpositive potential function condition in Lemma \ref{lemB2}, (\ref{eqLb}) become 

$$\int_{M}f|\mathring{Ric}|^{2}dM=\frac{1}{|\nabla f|}\int_{\partial M}\mathring{Ric}(\nabla f,\nabla f)dS.$$  
\end{remark}

Proceeding, we use the Gauss Equation (\ref{eqgauss}) to obtain the following lemma.

\begin{lemma}\label{lemRicBoundary}
Let $(M^{n},\,g,\,f)$ be a compact, oriented, connected Miao-Tam critical metric with smooth boundary $\partial M.$ Consider $e_n=-\frac{\nabla f}{|\nabla f|}$, and pick any orthonormal frame $\{e_1,...,e_{n-1}\}$
tangent to the level surface $\partial M.$ Then, on $\partial M,$ we have:
\begin{equation*}
R^{\partial M}_{ij}=R_{ij}-R_{injn}+\frac{n-2}{(n-1)^2|\nabla f|^2}g_{ij},
\end{equation*} where $Ric^{\partial M}$ is the Ricci tensor of $(\partial M, g_{|_{\partial M}}).$
\end{lemma}
\begin{proof}
By Gauss Equation (\ref{eqgauss}), for $1\leq i,j\leq n-1,$ we get
\begin{eqnarray*}
R^{\partial M}_{ij}&=&g^{kl}R^{\partial M}_{ikjl}\\
 &=&g^{kl}\left(R_{ikjl}-h_{il}h_{kj}+h_{ij}h_{kl}\right).
\end{eqnarray*} Whence, we use (\ref{segform}) to infer
\begin{eqnarray*}
R^{\partial M}_{ij}&=&R_{ij}-R_{injn}-\frac{1}{(n-1)^2|\nabla f|^2}g^{kl}g_{il}g_{kj}+\frac{1}{(n-1)^2|\nabla f|^2}g^{kl}g_{ij}g_{kl}\\
 &=&R_{ij}-R_{ninj}-\frac{1}{(n-1)^2|\nabla f|^2}g_{ij}+\frac{1}{(n-1)|\nabla f|^2}g_{ij}\\
 &=&R_{ij}-R_{injn}+\frac{n-2}{(n-1)^2|\nabla f|^2}g_{ij},
\end{eqnarray*} as we wanted to prove.
\end{proof}

We remember that the standard metrics on geodesic balls in space forms are Miao-Tam critical metrics. So, it is natural to ask if they are the only examples taking into account the boundary condition. In this sense, Miao and Tam obtained a classification for critical metrics of the volume functional on manifolds with connected boundary and nonpositive sectional curvature (cf. Corollary 3 in \cite{miaotam}). In the sequel, we obtain a general classification for Miao-Tam critical metrics without assumption on sectional curvature, which can be compared with Corollary 3 in \cite{miaotam}.

\begin{proposition}\label{propC}
Let $(M^n,\,g,\,f)$ be a connected Miao-Tam critical metric with connected boundary $\partial M$ and nonnegative potential function. Suppose that the scalar curvature of $(\partial M, g_{|_{\partial M}})$ is constant. Then $\mathring{Ric}(\nabla f,\nabla f)$ is a nonpositive constant along $\partial M.$ In particular, $\mathring{Ric}(\nabla f,\nabla f)=0$ if and only if $(M^n,\,g)$ is isometric to a geodesic ball in a simply connected space form $\Bbb{R}^{n},$ $\Bbb{H}^{n},$ or $\Bbb{S}^n.$
\end{proposition}

\begin{proof} Taking the trace in Lemma \ref{lemRicBoundary} we achieve
\begin{eqnarray*}
R^{\partial M}&=&g^{ij}R^{\partial M}_{ij}\\
 &=&g^{ij}\left(R_{ij}-R_{injn}+\frac{n-2}{(n-1)^2|\nabla f|^2}g_{ij}\right)\\
 &=&R-R_{nn}-R_{nn}+\frac{n-2}{(n-1)|\nabla f|^2},
\end{eqnarray*} where $R^{\partial M}$ stands for the scalar curvature of $(\partial M,g_{|_{\partial M}}).$ From this it follows that $$2R_{nn}+R^{\partial M}=R+\frac{n-2}{(n-1)|\nabla f|^2}.$$ See also Eq. (45) in \cite{miaotam}. 

Next, since $R^{\partial M}$ and $|\nabla f|$ are constant along $\partial M$ we deduce that $R_{nn}$ is constant along $\partial M.$ Of which we deduce that $Ric(\nabla f,\nabla f)$ is also constant along $\partial M.$ Taking into account that $M$ has constant scalar curvature, we immediately conclude that $\mathring{Ric}(\nabla f,\nabla f)$ is constant along $\partial M.$

Proceeding, we use Lemma \ref{lemB2} to infer
\begin{eqnarray*}
0&\leq&\int_M f|\mathring{Ric}|^2dM\\
 &=&-\frac{1}{|\nabla f|}\int_{\partial M}\mathring{Ric}(\nabla f,\nabla f)dS\\
 &=&-\frac{1}{|\nabla f|}\mathring{Ric}(\nabla f,\nabla f)|\partial M|.
\end{eqnarray*} Whence, we get $\mathring{Ric}(\nabla f,\nabla f)\leq0.$

On the other hand, by supposing $\mathring{Ric}(\nabla f,\nabla f)=0,$ we can apply Lemma \ref{lemB2} to obtain $$\int_Mf|\mathring{Ric}|dM_g = 0.$$ Since $f\geq0$ we infer $\mathring{Ric}\equiv0$ 
on $M\setminus\partial M.$ Clearly, by continuity, $\mathring{Ric}\equiv0$ in $M.$ This implies that $(M^n,\,g)$ is Einstein. Now, it suffices to apply Theorem \ref{thmMT} (see Theorem 1.1 in \cite{miaotamTAMS}) to conclude that $(M^n,\,g)$ is isometric to a geodesic ball in a simply connected space form $\Bbb{R}^{n},$ $\Bbb{H}^{n},$ or $\Bbb{S}^n.$ 

The converse statement is straightforward. Thereby, we finish the proof of the proposition.
\end{proof}

\begin{remark}
From Remark \ref{r2}, if we replace the nonnegative potential function condition by the nonpositive potential function condition the conclusion of Proposition \ref{propC} is exactly the same
\end{remark}

\section{Proof of the main results}

\subsection{Proof of Theorem \ref{thmMain}}

\begin{proof}
First of all, we already know that $M$ has constant scalar curvature. Therefore, if $M$ has null scalar curvature we may use the weak maximum principle \cite{GT}, jointly with (\ref{eqtrace}), to conclude that the boundary condition $f^{-1}(0)=\partial M$ implies that $f$ is positive in $M$ (see also Theorem 7 in \cite{miaotam}). But, if $M$ has positive scalar curvature, the boundary condition implies that $f$ does not change of signal. So, we start assuming that $f$ is nonnegative. In that case, we may use Lemma \ref{lemB2} to arrive at

\begin{eqnarray}
\label{eqK}
\int_{M}f|\mathring{Ric}|^{2}dM&=&-\frac{1}{|\nabla f|}\int_{\partial M}\left(Ric(\nabla f,\nabla f)-\frac{R}{3}|\nabla f|^2\right)dS\nonumber\\
 &=&-\frac{1}{|\nabla f|}\int_{\partial M}Ric(\nabla f,\nabla f)dS+\frac{R}{3}|\nabla f||\partial M|.
\end{eqnarray}
Now by Gauss Equation (\ref{eqgauss}) as well as (\ref{segform}) we deduce
\begin{eqnarray}
\label{G2}K&=&R_{1212}-h_{12}h_{21}+h_{11}h_{22}\nonumber\\
 &=&R_{1212}+\frac{1}{4|\nabla f|^{2}},
\end{eqnarray} where $K$ is the sectional curvature of $\partial M.$ From this it follows that

\begin{equation*}
R_{11}=R_{1212}+R_{1313}=R_{1313}+K-\frac{1}{4|\nabla f|^{2}},
\end{equation*}

\begin{equation*}
R_{22}=R_{2323}+K-\frac{1}{4|\nabla f|^{2}}
\end{equation*} and

\begin{equation*}
R_{33}=R_{1313}+R_{2323}.
\end{equation*} Hence, we immediately obtain

\begin{eqnarray*}
R&=&R_{11}+R_{22}+R_{33}\\
 &=&2R_{1313}+2R_{2323}+2K-\frac{2}{4|\nabla f|^{2}},
\end{eqnarray*} which can rewritten succinctly as
\begin{equation}
\label{R1323}R_{1313}+R_{2323}=\frac{R}{2}-K+\frac{1}{4|\nabla f|^{2}}.
\end{equation}

Easily one verifies that 
\begin{eqnarray}
Ric(\nabla f,\nabla f)&=&|\nabla f|^{2}R_{33}\nonumber\\
 &=&|\nabla f|^{2}(R_{1313}+R_{2323}).
\end{eqnarray}  This combined with (\ref{R1323}) yields

\begin{equation}
\label{Rab}
Ric(\nabla f,\nabla f)=\frac{R}{2}|\nabla f|^{2}-K|\nabla f|^{2}+\frac{1}{4}.
\end{equation}

Upon integrating (\ref{Rab}) we get

\begin{equation}
\label{e34}
\frac{1}{|\nabla f|}\int_{\partial M}Ric(\nabla f,\nabla f) dS=\frac{R}{2}|\nabla f||\partial M|-|\nabla f|\int_{\partial M}K dS+\frac{1}{4|\nabla f|}|\partial M|.
\end{equation} Then, substituting (\ref{e34}) into (\ref{eqK}) we obtain

\begin{eqnarray*}
\int_{M}f|\mathring{Ric}|^{2}dM &=&-\frac{R}{2}|\nabla f||\partial M|+|\nabla f|\int_{\partial M}K dS\nonumber\\&&-\frac{1}{4|\nabla f|}|\partial M|+\frac{R}{3}|\nabla f||\partial M|.
\end{eqnarray*} From this it follows that

\begin{eqnarray}
\label{intK2}\int_{\partial M}KdS&=&\frac{1}{|\nabla f|}\int_{M}f|\mathring{Ric}|^{2}dM+\left(\frac{R}{6}+\frac{1}{4|\nabla f|^2}\right)|\partial M|\nonumber\\&\ge&\left(\frac{R}{6}+\frac{1}{4|\nabla f|^2}\right)|\partial M|.
\end{eqnarray} Thereby, since that the scalar curvature of $M$ is nonnegative we immediately deduce

$$\int_{\partial M}K dS>0.$$ Therefore, by using the Gauss-Bonnet theorem we conclude that $\partial M$ is a $2$-sphere.

Hereafter, we notice that $C(R)=\frac{R}{6}+\frac{1}{4|\nabla f|^2}$ is a positive constant in (\ref{intK2}). So, since $\partial M$ is a $2$-sphere we may use again Gauss-Bonnet theorem, together with (\ref{intK2}), to infer
\begin{equation*}
|\partial M|\leq\frac{4\pi}{C(R)}.
\end{equation*} In particular, from (\ref{intK2}) the equality holds if only if $\int_M f|\mathring{Ric}|^2 dM=0.$ This forces $(M^{3},\,g)$ to be Einstein. Now, we may invoke Theorem \ref{thmMT} (see also Theorem 1.1 in \cite{miaotamTAMS}) to conclude that $(M^3,\,g)$ is isometric to a geodesic ball in a simply connected space form $\Bbb{R}^{3}$ or $\Bbb{S}^3.$ This is what we wanted to prove. 

It remains to analyze the case when $f$ is nonpositive. In this case, from the proof of Lemma \ref{lemB2} (see also Remark \ref{r2}), it is not difficult to show that

$$\int_{M}f|\mathring{Ric}|^{2}dM=\frac{1}{|\nabla f|}\int_{\partial M}Ric(\nabla f,\nabla f)dS -\frac{R}{3}|\nabla f||\partial M|.$$  

From now on the proof looks like that one of the previous case. In particular, we arrive at 

\begin{equation}
\int_{\partial M}KdS=-\frac{1}{|\nabla f|}\int_{M}f|\mathring{Ric}|^{2}dM+\left(\frac{R}{6}+\frac{1}{4|\nabla f|^2}\right)|\partial M|.
\end{equation} Taking into account that $f$ is nonpositive we achieve

$$\label{intK}\int_{\partial M}K dS\ge\left(\frac{R}{6}+\frac{1}{4|\nabla f|^2}\right)|\partial M|.$$ To conclude, it suffices to follow the arguments applied in the final steps of the first case.  So, the proof is completed.

\end{proof}

\subsection{Proof of Theorem \ref{thmtypebochener}}
\begin{proof}
We start recalling that, from Theorem 7 in \cite{miaotam}, $M^3$ has constant scalar curvature. Next, notice that

\begin{eqnarray}
\frac{1}{2}div(f\nabla|\ric|^2) & = & \frac{1}{2}\langle \nabla f, \nabla |\mathring{Ric}|^2 \rangle + \frac{f}{2}\Delta |\mathring{Ric}|^2.
\end{eqnarray} It is not difficult to check that $$\langle \nabla f, \nabla |\mathring{Ric}|^{2}\rangle=2\langle \nabla_{\nabla f}\mathring{Ric},\mathring{Ric}\rangle$$ and 

\begin{eqnarray*}
\Delta |\mathring{Ric}|^{2}&=&g^{ik}\nabla_{i}\nabla_{k}\big(\langle\mathring{Ric},\mathring{Ric}\rangle\big)\nonumber\\&=&2g^{ik}\big(\langle \nabla_{i}\nabla_{k}\mathring{Ric},\mathring{Ric}\rangle+\langle \nabla_{k}\mathring{Ric},\nabla_{i}\mathring{Ric}\rangle\big).
\end{eqnarray*} Of which we deduce
\begin{equation}
\label{eqtypebochener}
\frac{1}{2}div(f\nabla|\ric|^2) =\langle \nabla_{\nabla f}\mathring{Ric},\mathring{Ric}\rangle+f|\nabla \mathring{Ric}|^{2}+f\langle\Delta \mathring{Ric},\mathring{Ric}\rangle.
\end{equation}

On the other hand, computing the Laplacian of (\ref{IdRicHess}) we obtain
\begin{equation}
(\Delta \hess)_{ij} = (\Delta f)\rc_{ij} + f(\Delta \ric)_{ij} + 2 \langle \nabla \rc_{ij}, \nabla f \rangle.
\end{equation} From this it follows that
\begin{equation}\label{eqlaplacric}
f(\Delta \ric)_{ij} = (\Delta \hess)_{ij} + \frac{R}{2}f\rc_{ij} + \frac{3}{2}\rc_{ij} - 2 \langle \nabla \rc_{ij}, \nabla f \rangle.
\end{equation}

We remember that in dimension three $W\equiv 0.$ This data substituted into (\ref{weyl}) yields

\begin{equation}\label{curvature}
R_{ikjp} = (R_{ij}g_{kp} + R_{kp}g_{ij} - R_{ip}g_{kj} - R_{kj}g_{ip}) - \frac{R}{2}(g_{ij}g_{kp} - g_{ip}g_{kj}).
\end{equation}

Next, we compute the value of $(\Delta \hess)_{ij}.$ To do so, we use Lemma \ref{commutator} to obtain
\begin{eqnarray}
 \label{laplacianohess}
(\Delta \nabla^2 f)_{ij} & = & \nabla^2_{ij} \Delta f + (R_{jp}g_{ik} + R_{ip}g_{jk} - 2R_{ikjp})\nabla^k\nabla^p f \nonumber\\&&+ (\nabla_i R_{jp} + \nabla_j R_{pi} - \nabla_p R_{ij})\nabla^p f \nonumber\\&=& -\frac{R}{2}\nabla_i\nabla_j f + (\nabla_i R_{jp} + \nabla_jR_{pi} - \nabla_pR_{ij}) \nabla^pf + II,
\end{eqnarray} 
where
\begin{eqnarray}
\label{eqIImagic}
II & = & (R_{jp}g_{ik} + R_{ip}g_{jk} - 2R_{ikjp})\nabla^k\nabla^p f.
\end{eqnarray}

We shall treat $II$ separately. Indeed, substituting (\ref{eqMiaoTam1}) into (\ref{eqIImagic}) we arrive at

\begin{eqnarray*}
II & = & fR_{jp}R_{i}^{p}+fR_{ip}R_{j}^{p}-2fR^{kp}R_{ikjp}.
\end{eqnarray*} This combined with (\ref{curvature}) yields

\begin{eqnarray*}
II & = &  fR_{jp}R_{i}^{p}+fR_{ip}R_{j}^{p}\nonumber\\&&-2fR^{kp}\big(R_{ij}g_{kp} + R_{kp}g_{ij} - R_{ip}g_{kj} - R_{kj}g_{ip}- \frac{R}{2}g_{ij}g_{kp} + \frac{R}{2}g_{ip}g_{kj}\big).
\end{eqnarray*} Rearranging the terms we infer

\begin{eqnarray}
\label{eqhjg}
II & = & 3fR_{jp}R_{i}^{p}+3fR_{ip}R_{j}^{p}-3fRR_{ij}-2f|Ric|^{2}g_{ij}+fR^{2}g_{ij}.
\end{eqnarray} Now, substituting the traceless Ricci tensor $\mathring{Ric}=Ric-\frac{R}{3}g$ into (\ref{eqhjg}) we achieve

\begin{eqnarray*}
II & = & 3f\big(\rc_{jp}+\frac{R}{3}g_{jp}\big)\big(\rc_{i}^{p}+\frac{R}{3}g_{i}^{p}\big)+ 3f\big(\rc_{ip}+\frac{R}{3}g_{ip}\big)\big(\rc_{j}^{p}+\frac{R}{3}g_{j}^{p}\big)\nonumber\\&&-3fR\big(\rc_{ij}+\frac{R}{3}g_{ij}\big)-2f|Ric|^{2}g_{ij}+fR^{2}g_{ij}\nonumber\\&=&3f\rc_{jp}\rc_i^p + 3f\rc_{ip}\rc_j^p + R f \rc_{ij}+\frac{2}{3}fR^{2}g_{ij} -2f|Ric|^2 g_{ij}.
\end{eqnarray*} Hence, taking into account that $|\mathring{Ric}|^{2}=|Ric|^{2}-\frac{R^{2}}{3},$ we may written $II$ succinctly as

\begin{eqnarray*}
II & = & 3f\rc_{jp}\rc_i^p + 3f\rc_{ip}\rc_j^p + R f \rc_{ij} -2f|\ric|^2 g_{ij}.
\end{eqnarray*}

At the same time, returning to Eq. (\ref{laplacianohess}) we may write

\begin{eqnarray*}
(\Delta \hess)_{ij} & = & (\Delta \nabla^2 f)_{ij} - \frac{1}{3}\Delta(\Delta f)g_{ij} \\
& = & -\frac{R}{2}\nabla_i\nabla_j f + 3f(\rc_{jp}\rc^p_i + \rc_{ip}\rc^p_j) + R f\rc_{ij} \\
& & - 2f|\ric|^2g_{ij} + (\nabla_i \rc_{jp} + \nabla_j \rc_{pi} - \nabla_p \rc_{ij})\nabla^p f \nonumber\\&&- \frac{1}{3}\Delta(-\frac{R}{2} f -\frac{3}{2}) g_{ij},
\end{eqnarray*} and using that $\mathring{Hess}f=Hess\,f-\frac{\Delta f}{3}g$ we immediately obtain
\begin{eqnarray*}
(\Delta \hess)_{ij}& = & -\frac{R}{2}((\hess)_{ij} + \frac{\Delta f}{3})g_{ij}+ 3f(\rc_{jp}\rc^p_i + \rc_{ip}\rc^p_j) +Rf\rc_{ij} -2f|\ric|^2g_{ij}  \\
& & + (\nabla_i \rc_{jp} + \nabla_j \rc_{pi} - \nabla_p \rc_{ij})\nabla^p f + \frac{R}{6}\Delta f g_{ij}.
\end{eqnarray*} Next, taking into account that $f\mathring{Ric}=\mathring{Hess} f$ we get
\begin{eqnarray}\label{laplacianhesswithouttrace}
(\Delta \hess)_{ij} &=& \frac{R}{2}f \rc_{ij} + 3f(\rc_{jp}\rc^p_i + \rc_{ip}\rc^p_j) -2f|\ric|^2g_{ij}\nonumber\\&& + (\nabla_i \rc_{jp} + \nabla_j \rc_{pi} - \nabla_p \rc_{ij})\nabla^p f.
\end{eqnarray} Therefore, plugging (\ref{laplacianhesswithouttrace}) into (\ref{eqlaplacric}) we arrive at

\begin{eqnarray*}
f(\Delta \ric)_{ij} & = & (\Delta \hess)_{ij} + \frac{R}{2}f\rc_{ij} + \frac{3}{2}\rc_{ij} - 2 \langle \nabla \rc_{ij}, \nabla f \rangle\\
& = &   Rf\rc_{ij} + \frac{3}{2}\rc_{ij}  -2f|\ric|^2g_{ij} + 3f(\rc_{jp}\rc^p_i + \rc_{ip}\rc^p_j)\\
& & + (\nabla_i \rc_{jp} + \nabla_j \rc_{pi} - \nabla_p \rc_{ij})\nabla^p f - 2\nabla_p\rc_{ij}\nabla^p f.
\end{eqnarray*} This implies that

\begin{eqnarray}
\label{34g}
f(\Delta \ric)_{ij} & = & R f\rc_{ij} + \frac{3}{2}\rc_{ij} -2f|\ric|^2g_{ij} + 3f(\rc_{jp}\rc^p_i + \rc_{ip}\rc^p_j)\nonumber\\&& + (\nabla_i \rc_{jp} + \nabla_j \rc_{pi} -3 \nabla_p \rc_{ij})\nabla^p f.
\end{eqnarray}

Returning to Eq. (\ref{eqtypebochener}), we may use (\ref{34g}) to achieve

\begin{eqnarray*}
\frac{1}{2}div(f\nabla|\ric|^2) & = & (\nabla_k \rc_{ij}) \rc^{ij} \nabla^k f + f |\nabla \ric|^2 \\
& & +\Big[Rf\rc_{ij} + \frac{3}{2}\rc_{ij} - 2f|\ric|^2 g_{ij} + 3f(\rc_{jp}\rc^p_i + \rc_{ip}\rc^p_j)\nonumber\\&& + (\nabla_i \rc_{jp} + \nabla_j \rc_{pi} -3 \nabla_p \rc_{ij})\nabla^p f \Big]\rc^{ij} \\
& = & -2(\nabla_p \rc_{ij}) \rc^{ij} \nabla^p f + 2(\nabla_i \rc_{jp}) \rc^{ij} \nabla^p f \nonumber\\&& + f |\nabla \ric|^2 + Rf|\ric|^2 + \frac{3}{2}|\ric|^2 + 6f tr(\ric^3)\\
& = & 2(\nabla_i \rc_{pj} - \nabla_p \rc_{ij}) \rc^{ij} \nabla^p f + f |\nabla \ric|^2\nonumber\\&& + Rf|\ric|^2 + \frac{3}{2}|\ric|^2 + 6f tr(\ric^3).
\end{eqnarray*} Thereby, since $M$ has constant scalar curvature, it follows from (\ref{cotton}) that

\begin{eqnarray*}
\frac{1}{2}div(f\nabla|\ric|^2) &=&  -2C_{pij}\rc^{ij}\nabla^pf + f |\nabla \ric|^2 + Rf|\ric|^2\nonumber\\&& + \frac{3}{2}|\ric|^2 + 6f tr(\ric^3).
\end{eqnarray*} Finally, it suffices to use Lemma \ref{lemacotton} to deduce

\begin{eqnarray}
\frac{1}{2}div(f\nabla|\ric|^2)& =& f \frac{|C|^2}{2} + f |\nabla \ric|^2 + Rf|\ric|^2 \nonumber\\&&+ \frac{3}{2}|\ric|^2 + 6f \big(tr(\ric^3)\big).
\end{eqnarray} This finishes the proof of the theorem.
\end{proof}

\subsection{Proof of Corollary \ref{corthmtypebochener}}
\begin{proof} We start remembering the classical Okumura's Lemma (cf. Lemma 2.1 in \cite{okumura}):
\begin{equation}
tr(\ric^3) \geq - \frac{1}{\sqrt{6}}|\ric|^3.
\end{equation} Therefore, we may use Theorem \ref{thmtypebochener} to infer
\begin{eqnarray*}
\frac{1}{2}div(f\nabla|\ric|^2) &\geq& \big(|\nabla \ric|^2 + \frac{|C|^2}{2}\big)f + \big(R - \sqrt{6}|\ric|\big)|\ric|^2 f \nonumber\\&&+ \frac{3}{2}|\ric|^2.
\end{eqnarray*} Upon integrating of the above expression over $M$ we use our assumption to deduce that $(M^3,\,g)$ is Einstein. Now, we are in position to apply Theorem \ref{thmMT} (cf. Theorem 1.1 in \cite{miaotamTAMS}) to conclude that $M^3$ is isometric to a geodesic ball in $\mathbb{S}^3.$ 

This is what we wanted to prove.
\end{proof}

\begin{acknowledgement}
The authors want to thank A. Barros for helpful conversations about this subject. Moreover, the authors want to thank the referee for his careful reading and valuable suggestions. 
\end{acknowledgement}

\end{document}